\documentclass[12pt]{article}
\usepackage{amsmath,amssymb, amsfonts, amsthm,amscd}
\usepackage[T2A]{fontenc}
\usepackage[cp1251]{inputenc}
\usepackage[russian,ukrainian,english]{babel}
\usepackage{graphicx}

\newtheorem{proposition}{Proposition}

\newtheorem{theorem}{Theorem}

\newtheorem*{proposition*}{Proposition}
\newtheorem*{lemma*}{Lemma}
\newtheorem*{theorem*}{Theorem}
\newtheorem*{corollary*}{Corollary}

\newtheorem{remark}{Remark}

\newtheorem*{definition*}{Definition}
\newtheorem*{remark*}{Remark}
\newtheorem*{example*}{Example}

\DeclareMathOperator{\rad}{rad}

\begin{document}
\begin{center}
\textbf{A commutative  Bezout domain in  which
every maximal ideal is principal is an elementary divisor ring}
\end{center}
\vskip 0.1cm \centerline{\textbf{Zabavsky Bogdan}}

\vskip 0.3cm

\centerline{\footnotesize{Department of Algebra and Logic,  Ivan Franko National University of L'viv,  Ukraine}}
\vskip 0.5cm

\centerline{\footnotesize{October, 2012}}
\vskip 0.7cm

\footnotesize{\noindent\textbf{Abstract:} \textit{In this article we revisit a problem regarding Bezout domains,
namely, whether every Bezout domain is an elementary divisor domain.
We prove that a Bezout domain in which every maximal ideal is principal is
an elementary divisor ring.}

}

\vskip 1cm

\normalsize

All rings considered will be commutative and have the identity. A ring is
a Bezout ring if every its finitely generated ideal is a principal.

Kaplansky \cite{item1} defined the class of elementary divisor rings as those rings
$R$ for which every matrix $M$ over $R$ admits a diagonal reduction, that is
there exist invertible (unimodular) matrices $P$ and $Q$ such that $PMQ$ is a
diagonal matrix $D = (d_{ii})$, which the property that every $d_{ii}$ is a divisor of
$d_{i+1,i+1}$. He showed that if $R$ is an elementary divisor domain, then every
finitely presented module over $R$ is a direct sum of cyclic modules. It was
later shown in \cite{item2} that the converse is true answering a question of Warfield~\cite{item3}.

A ring $R$ is fractionaly regular if for every non-zero element $a$ from $R$
the classical quotient ring $Q_{cl}({}^{R}\!/\!_{\rad(aR)})$ is regular \cite{item4}. We say the ring $R$ has stable range 2 if whenever
$aR+bR+cR=R$, then there are $x, y\in R$ such that $(a+cx)R+(b+cy)R = R$
 \cite{item4}.

We say $R$ is a reduced ring if it has no nilpotent elements other than $0$.
Of course, this is equivalent to saying that the intersection of the minimal
prime ideals of $R$ is $0$.

For every ideal $I$ in $R$ we define the annihilator of $I$ by $I^\perp = \{x\in R \mid
ix=0, \forall  i\in I \}$. The nilradical of $R$ is denoted by $\rad(R)$ and we denote by
$Q_{cl}(R)$ the classical ring of quotients of $R$.

Following Faith  \cite{item5} a ring $R$ is zip if $I$ is an ideal and if $I^\perp = 0$ then
$I_1^\perp=0$ for a finitely generated ideal $I_1\subseteq I$. We introduce the concept of
mzip ring as a ring with the property: if $M$ is a maximal ideal and $M^\perp = 0$
then $M_1^\perp = 0$ for a finitely generated ideal $M_1\subseteq M$. Every zip ring is a
mzip ring and if every maximal ideal of Bezout ring is finitely generated, then ring
is a mzip ring.

An ideal $I$ of a ring $R$ is dense if it is an annihilator is zero. Thus $I$ is a
dense is and only if it is a faithfull $R$-module. If every dense maximal ideal
of $R$ contains a regular element then $R$ is evidently mzip.

Actually every dense maximal ideal $R$ contains a regular element if for
every maximal ideal of $Q_{cl}(R)$ the classical quotient ring is an annulet. The
rings with the property that maximal ideals are annulets are called Kasch
rings  \cite{item6}. A ring $R$ is McCoy if every finitely generated dense ideal contains
a regular element (In  \cite{item7} this is called "Property A").

Using the triviality that $I$ is an ideal of $Q_{cl}(R)$ and $I = I_0Q_{cl}(R)$, where
$I_0$ is annulet of $R$ (and in this case $I_0 = I\cap R$), one sees that $R$ is McCoy
if $Q_{cl}(R)$ is McCoy.

If $I$ is a finitely generated dense ideal of a Bezout ring $R$, then $I = aR$
for $a\in  R$, so $a \in  I$ is a regular element. Then a Bezout ring  is  a McCoy
ring and we obtain the following result:

\begin{proposition}\label{prop1} Let $R$ be  a mzip Bezout ring, then every dense
maximal ideal contains a regular element and $Q_{cl}(R)$ is a Bezout Kasch ring.
\end{proposition}

\begin{proposition}
Let $R$ be a reduced Bezout ring in which every maximal
ideal is projective. Then $R$ is an mzip if and only if every maximal ideal of
$R$ is principal.
\end{proposition}

\begin{proof} If $M$ is a dense maximal ideal of $R$ and $R$ is mzip, then exists a
principal ideal $M_1 = mR$ such that $M_1^\perp = 0$ and $M_1 \subset M$. Obviously, $m$
is a regular element. Since $m \in M$ this is possible when $M$ is a
principal ideal~\cite{item10}. The proposition is proved.
\end{proof}

\begin{theorem} Let $R$ be a Bezout domain and for every non-zero element
$a \in R$ the factor-ring ${}^{R}\!/\!_{\rad(aR)}$ is mzip. Then $R$ is an elementary divisor
ring.
\end{theorem}

\begin{proof} Denote ${}^{R}\!/\!_{aR}=\overline{R}$ and $K={}^{\overline{R}}\!/\!_{\rad(\overline{R})}$. Obviously, $K$ is a reduced
ring. By  \cite{item8}, $K$ is a ring with projective socle. Then for every maximal ideal
$M$ of $K$, we have $M^\perp = 0$ or $M = eK$, where $e^2 = e\in K$. By Proposition~1, every dense maximal ideal contains a regular element. Obviously $Q_{cl}(K)$
is a Kasch reduced ring and by  \cite{item7}, $Q_{cl}(K)$ is a regular ring. Then $R$ is a
fractionaly regular ring of stable range 2. By  \cite{item4} and  \cite{item11}, $R$ is an elementary divisor
ring. The theorem is proved.
\end{proof}

Is obviously the following theorem:

\begin{theorem}\label{theor2} Let $R$ be a Bezout domain in which every maximal ideal is
principal. Then $R$ is an elementary divisor ring.
\end{theorem}

\begin{remark} It is easy to deduce that a reduced ring $R$ has Kasch $Q_{cl}(R)$ iff $Q_{cl}(R)$ is semisimple, i.e. it is a finite product of fields. Thus, a reduced $R$ with Kasch $Q_{cl}(R)$ has a finite number of  minimal prime ideals.
\end{remark}

By proposition~\ref{prop1} and theorem~\ref{theor2} we have.

\begin{theorem} Let $R$ be a Bezout domain in which every maximal ideal is principal. Then for every nonzero element $a\in R$ there are only finitely many prime ideals minimal over $a$.
\end{theorem}

By \cite{item12} we have

\begin{theorem} Let $R$ be a Bezout domain in which every maximal ideal is principal. Then each nonzero principal ideal $aR$ of $R$ can be written in the form $aR=P_1P_2\dots P_n$, where each $P_i$ has prime radical and the $P_i$ are pairwise comaximal.
\end{theorem}

\end{document}